\documentclass[letterpaper, 10 pt, conference]{ieeeconf}

\IEEEoverridecommandlockouts
\usepackage{amsmath} 
\usepackage{amssymb}  
\usepackage[utf8]{inputenc}
\usepackage{color}
\usepackage{tikz}
\usetikzlibrary{shapes,arrows.meta}
\usepackage{xspace}
\usepackage{siunitx}
\usepackage{cite}
\usepackage{bm,fixmath}

\usepackage{graphicx}

\usepackage{algorithm}
\usepackage{algpseudocode}

\algrenewcommand\algorithmicrequire{\textbf{Input:}}

\DeclareMathOperator*{\argmin}{arg\,min}
\DeclareMathOperator{\prox}{prox}
\DeclareMathOperator{\refl}{refl}

\newtheorem{definition}{Definition}
\newtheorem{remark}{Remark}

\newtheorem{theorem}{Theorem}
\newtheorem{assumption}{Assumption}
\newtheorem{lemma}{Lemma}

\newcommand{\norm}[1]{\left\lVert#1\right\rVert}
\newcommand{\x}{\mathbold{x}}
\newcommand{\y}{\mathbold{y}}
\newcommand{\w}{\mathbold{w}}
\newcommand{\z}{\mathbold{z}}
\newcommand{\p}{\mathbold{p}}
\newcommand{\R}{\mathbb{R}}
\newcommand{\A}{\mathbold{A}}
\newcommand{\bv}{\mathbold{b}}
\renewcommand{\baselinestretch}{0.96}

\def\endstatement{\hfill$\square$}

\title{\LARGE \bf
Prediction-Correction Splittings \\ for Nonsmooth Time-Varying Optimization
\vskip-3mm}
\author{Nicola Bastianello, Andrea Simonetto, Ruggero Carli
\thanks{N. Bastianello and R. Carli are with the Department of Information Engineering, University of Padova, Italy.\newline {\tt\scriptsize nicola.bastianello.3@phd.unipd.it, carlirug@dei.unipd.it}. \newline A. Simonetto is with IBM Research Ireland. Dublin, Ireland. \newline{\tt\scriptsize andrea.simonetto@ibm.com}.}
}

\begin{document}

\maketitle
\thispagestyle{empty}
\pagestyle{empty}

\begin{abstract}
We address the solution of time-varying optimization problems characterized by the sum of a time-varying strongly convex function and a time-invariant nonsmooth convex function.
We design an online algorithmic framework based on prediction-correction, which employs splitting methods to solve the sampled instances of the time-varying problem.
We describe the prediction-correction scheme and two splitting methods, the forward-backward and the Douglas-Rachford. Then by using a result for generalized equations, we prove convergence of the generated sequence of approximate optimizers to a neighborhood of the optimal solution trajectory. Simulation results for a leader following formation in robotics assess the performance of the proposed algorithm.
\end{abstract}

\begin{keywords}
time-varying optimization, prediction-correction, splitting methods, forward-backward, Douglas-Rachford, generalized equations
\end{keywords}

\section{Introduction}\label{sec:intro}
We are interested in the solution of time-varying optimization problems in the form
\begin{equation}\label{eq:tv-problem}
	\x^*(t) = \argmin_{\x\in\R^n} \{ f(\x;t) + g(\x) \}
\end{equation}
where $f:\R^n \times \R_+ \to \R$ is smooth and strongly convex, and $g:\R^n \to \R \cup \{+\infty\}$ is proper, closed and convex, but possibly non-differentiable. Since the solution $\x^*(t)$ -- the trajectory -- changes over time, our objective is to ``track'' such trajectory up to an error bound.

Problems conforming to the framework of~\eqref{eq:tv-problem} arise in many applications. The powerful and widely used \textit{model predictive control} (MPC) requires that we solve an optimization problem that varies with time \cite{jerez2014embedded,hours2016parametric,paternain2018} in order to compute the control action. In signal processing, the estimation of time-varying signals on the basis of observations gathered online can be cast as the problem of solving a series of varying problems \cite{asif2010dynamic,asif2014sparse,charles2016dynamic}. In robotics, path tracking and leader following problems can be cast in the framework of \eqref{eq:tv-problem}, see for example \cite{verscheure2009time,ardeshiri2011convex,dixit2018online}. Other application domains are economics~\cite{dontchev2013euler}, smart grids~\cite{Zhou2018}, and non-linear optimization~\cite{Zavala2010, Dinh2012}. 

In this paper, we are interested in solving problem~\eqref{eq:tv-problem} in a discrete-time framework, which is amenable to direct implementation with digital hardware. In particular, we discretize the problem with a sampling period $T_{\mathrm{s}}:=t_{k+1}-t_k$, which results in the sequence of time-invariant problems
\begin{equation}\label{eq:ti-problem}
	\x^*(t_k) = \argmin_{\x\in\R^n} \{ f(\x;t_k) + g(\x) \}, \quad k \in \mathbb{N}.
\end{equation}
Due to the time-varying nature of the problem, the sequence of problems is not known ahead of time, and each problem $k$ is revealed when sampled at $t_k$. 

The smaller the sampling time, the higher the accuracy of the solution computed will be. However, the time-invariant problems may be too difficult to solve inside a sampling period, hence there exists a natural trade-off between accuracy and practical implementation constraints.

We are interested in the solution of~\eqref{eq:ti-problem} using a \emph{prediction-correction} scheme. Prediction-correction algorithms have been studied both in the discrete-time framework that we employ \cite{simonetto2016class,simonetto2017prediction,simonetto2018tac} and in a continuous-time setup \cite{fazlyab2017prediction,Rahili2015,fazlyab2016self}.

In the cited works, however, the proposed algorithms can handle only smooth optimization problems. Our goal in this paper is to extend the proposed approach to the more general nonsmooth \eqref{eq:ti-problem}. In particular, in \cite{simonetto2017prediction} $g(\x)$ could be interpreted as the indicator function of a convex set, while in this work we are interested in accepting general, nonsmooth convex functions.

While the prediction-correction scheme will be described in detail in the following section, we mention here that two optimization problems will be needed to be solved at each time $t_k$. The solution of these problems will be obtained by using \emph{splitting methods}, and in particular the forward-backward and the Douglas-Rachford splitting. For background and review of these methods we refer the interested reader to \cite{ryu2016primer,bauschke2017convex}, and for some examples of application to convex optimization problems \cite{combettes2005signal,combettes2011proximal}.

\paragraph*{Organization} The paper is structured as follows. Section~\ref{sec:pred-corr} introduces the proposed prediction-correction scheme with splitting methods. Section~\ref{sec:convergence} states the convergence results for the proposed algorithm, and reports a theoretical finding instrumental in proving them. Section~\ref{sec:simulation} describes the numerical simulations carried out to evaluate the performance of the algorithm. Section~\ref{sec:conclusions} concludes the paper. Proofs can be found in the Appendix.

\paragraph*{Basic definitions} We say that a function $\varphi:\R^n \to \R$ is $m$-\textit{strongly convex} for a constant $m \in \R_+$ iff $\varphi(\x) - \frac{m}{2}\|\x\|^2$ is convex. The function $\varphi$ is said to be $L$-\textit{smooth} if its gradient is $L$-Lipschitz continuous, or equivalently $\varphi(\x) - \frac{L}{2}\|\x\|^2$ is concave. We denote the class of $m$-strongly convex and $L$-smooth functions with $\mathcal{S}_{m,L}(\R^n)$.

A function $\varphi:\R^n \to \R \cup \{+\infty\}$ is said to be \textit{closed} if for any $a \in \R$ the set $\{ \x \in \operatorname{dom}(f)\ |\ \varphi(\x) \leq a \}$ is closed. A function is said to be \textit{proper} if it does not attain $-\infty$. We denote the class of closed, convex, and proper functions with $\Gamma_0(\R^n)$.

Given $\varphi \in \Gamma_0(\R^n)$ we define its \textit{subdifferential} as the set-valued operator $\partial \varphi: \R^n \rightrightarrows \R^n$ such that
$$
	\x \mapsto \left\{ \z \in \R^n\ |\ \forall \y \in \R^n:\ \langle \y-\x, \z \rangle + \varphi(\x) \leq \varphi(\y) \right\}.
$$

\section{Prediction-Correction and Splittings}\label{sec:pred-corr}
We describe the prediction-correction scheme that will be applied to solve \eqref{eq:ti-problem}, and we introduce the splitting methods that will be employed in combination with it. We will use the following assumptions.

\begin{assumption}\label{as:first}
The function $f:\R^n \times \R_+ \to \R$ belongs to $\mathcal{S}_{m,L}(\R^n)$ uniformly in time.  The function $g:\R^n \to \R \cup \{+\infty\}$ belongs to $\Gamma_0(\R^n)$ and it is in general nonsmooth. 
\end{assumption}

\begin{assumption}\label{as:first-bis}
The function $f$ has bounded time derivative of its gradient derivative as: 
$\norm{\nabla_{t\x} f(\x;t)} \leq C_0$.
\end{assumption}

\begin{assumption}\label{as:second}
The function $f$ is at least three times differentiable and has bounded derivatives w.r.t. $\x \in \R^n$ and $t \in \R_+$ as: 
\begin{align*}
	\norm{\nabla_{\x\x\x} f(\x;t)} & \leq C_1, \quad \norm{\nabla_{\x t \x} f(\x;t)} \leq C_2, \\
	& \norm{\nabla_{tt \x} f(\x;t)} \leq C_3.
\end{align*}
\end{assumption}

\smallskip

Assumption~\ref{as:first} ensures by strong convexity that the solution to the problem is unique at each time, and that the gradient of $f$ is Lipschitz continuous. These conditions are common for time-varying scenarios, \textit{e.g.}, see \cite{simonetto2016class,dontchev2013euler}. Assumption~\ref{as:first-bis} further guarantees that the temporal variability of the gradient of $f$ is bounded, and thus that the problem does not vary too widely between time instants. Finally, Assumption~\ref{as:second} -- which is not strictly necessary to prove convergence of our algorithm -- imposes boundedness of the tensor $\nabla_{\x\x\x}f(\x;t)$, which is a typical assumption for second-order algorithms. Moreover, it bounds the variability of the Hessian of $f$ over time, which guarantees the possibility of performing more accurate predictions of the optimal trajectory.

\subsection{Prediction-correction}
The scheme is characterized by two phases, the \emph{prediction} step, which at time $t_k$ computes an approximate solution to the problem at time $t_{k+1}$; and the \emph{correction} step, during which the predicted solution is refined using the information that becomes available at time $t_{k+1}$.

The problem that we are interested to approximately solve during the prediction step is $\x^*(t_{k+1}) = \argmin_\x \{ f(\x;t_{k+1}) + g(\x) \}$, which is equivalently characterized by the generalized equation
\begin{equation}\label{eq:gen-equation}
	\nabla_{\x} f(\x^*(t_{k+1});t_{k+1}) + \partial g(\x^*(t_{k+1})) \ni 0.
\end{equation}
However, at time $t_k$ the function $f(\x;t_{k+1})$ is not known and therefore we need to approximate it using the information available, \textit{i.e.}, $f(\x;t_k)$ and the current state $\x(t_k) =: \x_k$. This setup models for instance problems in which the cost function depends on observations that are gathered online.

In order to compute the prediction -- that we will denote $\tilde{\x}_{k+1|k}$ -- we approximate \textit{the gradient} of $f(\x;t_{k+1})$ by using a Taylor expansion around $(\x_k; t_k)$, which is given by
\begin{align}
\begin{split}
	\nabla h_k(\x) &= \nabla_{\x} f(\x_k;t_k) + \\ & + \nabla_{\x\x} f(\x_k;t_k) (\x - \x_k) + T_{\mathrm{s}} \nabla_{t\x} f(\x_k;t_k)
\end{split}
\end{align}
where the subscript indicates that $h$ is computed with information available at time $t_k$.

\begin{remark}
By Assumption~\ref{as:first}, it follows that $h_k$ belongs to $\mathcal{S}_{m,L}(\R^n)$, since $\nabla^2 h_k(\x) = \nabla_{\x\x} f(\x_k;t_k)$.
\end{remark}

An approximate solution to \eqref{eq:gen-equation} can now be computed solving the new generalized equation
\begin{equation}\label{eq:approx-gen-equation}
	\nabla h_k({\x}_{k+1|k}) + \partial g({\x}_{k+1|k}) \ni 0.
\end{equation}
Thus the prediction step requires the solution of this approximated generalized equation with initial condition $\x_k$.

At time $t_{k+1}$ the new information $f(\x;t_{k+1})$ is made available and therefore the correction step can be computed. The correction step approximately solves~\eqref{eq:gen-equation} by using the prediction as initial condition.

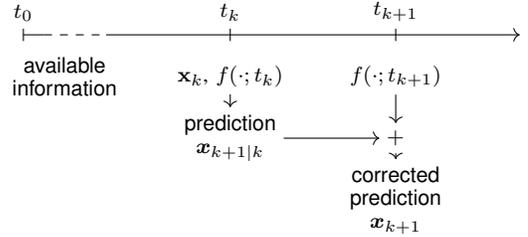
\begin{figure}[!ht]
\centering
\footnotesize
	\begin{tikzpicture}[scale=0.55]
	
		\draw (-3,0) -- (-2.5,0);
		\draw[dashed] (-2.5,0) -- (-1,0);
		\draw[->] (-1,0) -- (9,0);
		\draw (-3,-0.15) -- (-3,0.15) node[above] {$t_0$};
		\draw (2,-0.15) -- (2,0.15) node[above] {$t_{k}$};
		\draw (6,-0.15) -- (6,0.15) node[above] {$t_{k+1}$};
		
		\node[align=center] at (-2,-1) {\textsf{available}\\\textsf{information}};
		\node (tk) at (2,-1) {$\mathbf{x}_k$, $f(\cdot;t_k)$};
		\node (tk1) at (6,-1) {$f(\cdot;t_{k+1})$};
		
		\node[align=center] (pr) at (2,-2.5) {\textsf{prediction}\\$\x_{k+1|k}$};
		\node[align=center] (co) at (6,-4) {\textsf{corrected}\\\textsf{prediction}\\$\x_{k+1}$};
		\node (plus) at (6,-2.5) {$+$};
		
		\path[->] (tk) edge (pr)
			  (pr) edge (plus)
			  (tk1) edge (plus)
			  (plus) edge (co);
			
	\end{tikzpicture}

\caption{The prediction-correction scheme.}
\label{fig:prediction-correction-scheme}
\end{figure}

\begin{remark}
In order to compute the approximation $\nabla h_k(\x)$ of the gradient of $f(\x;t_{k+1})$ we assumed to know the derivative $\nabla_{t\x} f(\x;t_k)$. However in practice this is not always possible, and thus we can think of approximating the derivative using a \textit{first-order backward finite difference}
$$
	\nabla_{t\x} f(\x_k;t_k) \simeq  \left( \nabla_{\x} f(\x_k;t_k) - \nabla_{\x} f(\x_k;t_{k-1}) \right)/{T_\mathrm{s}},
$$
which introduces an error bounded by $O(T_\mathrm{s}^2)$ \cite{simonetto2016class}.
\end{remark}

\subsection{Splitting methods}
We have defined two problems that need solving, represented by the generalized equations \eqref{eq:gen-equation} and \eqref{eq:approx-gen-equation}. In order to do so, we can apply splitting methods, that are well suited to minimizing the sum of two (possibly non-differentiable) convex functions. In the following we introduce some necessary background on operator theory and the two most widely used splitting methods, the forward-backward and the Douglas-Rachford. For a treatment of operator theory and its uses for convex optimization we refer the reader to~\cite{ryu2016primer,bauschke2017convex}.

\begin{definition}
A mapping $\mathcal{T}: \R^n \to \R^n$ is said to be \textit{nonexpansive} if it has unitary Lipschitz constant, that is if $\norm{\mathcal{T}\x - \mathcal{T}\y} \leq \norm{\x - \y}$ for any two $\x, \y \in \R^n$.
\end{definition}

\begin{definition}
Let $f:\R^n \to \R \cup \{+\infty\}$ be a closed, proper and convex function, and let $\rho > 0$. The corresponding \textit{proximal operator} is defined as
$$
	\prox_{\rho f}(\x) = \argmin_{\y\in\R^n} \left\{ f(\y) + \frac{1}{2\rho} \norm{\y - \x}^2 \right\},
$$
and the \textit{reflective operator} is $\refl_{\rho f}(\x) = 2\prox_{\rho f}(\x) - \x$.
\end{definition}

\begin{definition}
Let $\mathcal{T}:\R^n \to \R^n$, the \textit{fixed points} of this operator are all the points $\bar{\x} \in \R^n$ such that $\bar{\x} = \mathcal{T}\bar{\x}$.
\end{definition}

Operator theory can be employed to solve convex optimization problems by designing suitable operators whose fixed points are the solutions of the original problem. Consider in particular the problem $\min_\x \{ \varphi(\x) + \gamma(\x) \}$, where $\varphi:\R^n \to \R$ belongs to $\mathcal{S}_{m,L}(\R^n)$, while $\gamma:\R^n \to \R \cup \{+\infty\}$ belongs to $\Gamma_0(\R^n)$. This is equivalent to asking that we solve the generalized equation $\nabla \varphi(\x) + \partial \gamma(\x) \ni 0$\footnote{Where recall that $\gamma(\x)$ is not necessarily smooth.}. Notice that both \eqref{eq:gen-equation} and \eqref{eq:approx-gen-equation} conform to this class of problems.

Since the problem of interest has a separable structure, we introduce now two examples of \textit{splitting operators} that exploit this fact. The first is the \textit{forward-backward} (FB) operator, defined as
\begin{equation}
	\mathcal{T}_{\mathrm{FB}} = \prox_{\rho \gamma} \circ (I - \rho \nabla \varphi)
\end{equation}
where $\rho < 2/L$, and the second is the \textit{Douglas-Rachford} (DR)
\begin{equation}
	\mathcal{T}_{\mathrm{DR}} = \frac{1}{2} \left( I + \refl_{\rho \gamma} \circ \refl_{\rho \varphi} \right)
\end{equation}
where it must be $\rho > 0$.


In order to derive the fixed points of the operators above we can employ the \textit{Banach-Picard} iteration, which consists of the repeated application of the chosen operator \cite{bauschke2017convex}. The resulting algorithms are called \textit{splitting methods}, and in particular the \textit{forward-backward spitting} (FBS) is $\x(k+1) = \mathcal{T}_{\mathrm{FB}}\x(k)$, $k \in \mathbb{N}$, or equivalently
\begin{align}\label{eq:fbs}
\begin{split}
	& \y(k) = \x(k) - \rho \nabla \varphi(\x(k)) \\
	& \x(k+1) = \prox_{\rho \gamma}(\y(k)).
\end{split}
\end{align}
The fixed points of $\mathcal{T}_{FB}$ computed via FBS coincide with the solutions to the original problem $\min_\x \{ \varphi(\x) + \gamma(\x) \}$.

The \textit{Douglas-Rachford splitting} (DRS) is instead applied to the auxiliary variable $\z$ and not -- as is the FBS -- directly to $\x$. The DRS is thus described by the update $\z(k+1) = \mathcal{T}_{\mathrm{DR}}\z(k)$, $k \in \mathbb{N}$, or equivalently
\begin{align}\label{eq:drs}
\begin{split}
	& \x(k) = \prox_{\rho \varphi}(\z(k)) \\
	& \y(k) = \prox_{\rho \gamma}(2\x(k) - \z(k)) \\
	& \z(k+1) = \z(k) + \y(k) - \x(k)
\end{split}
\end{align}
where the additional variable $\y$ is introduce to break the computations in smaller pieces. Let now $\z^*$ be a fixed point of $\mathcal{T}_{DR}$ computed by the DRS, the corresponding solution to the original problem can be computed as $\x^* = \prox_{\rho \varphi}(\z^*)$.

Under the assumptions above, for the $\varphi$ and $\gamma$ functions, it is possible to prove that the two splitting methods converge to a fixed point of the respective operator, and thus that they reach a solution to the original problem.

In particular, for the FBS it is possible to show that \cite{taylor2017convex}
\begin{equation}\label{eq:convergence-splitting}
	\norm{\x(k+1) - \x^*} \leq \zeta_{\textrm{FB}} \norm{\x(k) - \x^*}.
\end{equation}
where
\begin{equation}\label{eq:zeta_fbs}
	\zeta_{\textrm{FB}} =  \max\{ |1-\rho m|, |1-\rho L| \},
\end{equation}
and recursively, $\norm{\x(k) - \x^*} \leq \zeta_{\textrm{FB}}^k \norm{\x(0) - \x^*}$ 

For the DRS, instead, the following Lemma holds.

\begin{lemma}\label{lem:convergence-drs}
The sequence $\{\x(k)\}_{k \in \mathbb{N}}$ generated by the DRS according to \eqref{eq:drs} satisfies the following inequality
\begin{equation}\label{eq:convergence-drs}
	\norm{\x(k) - \x^*} \leq {\zeta}^k_\textrm{DR} \frac{1+\rho L}{1+\rho m} \norm{\x(0) - \x^*},
\end{equation}		
\begin{equation}\label{eq:zeta_drs}
	{\zeta}_\textrm{DR} =  \max\Big\{ \frac{1}{1+\rho m}, \frac{\rho L}{1+\rho L} \Big\}.
\end{equation} \endstatement
\end{lemma}

\begin{proof}
Given in the Appendix.
\end{proof}

In the following, we will use $\zeta(k)$ as a function of $k$ to indicate either $\zeta^k_{\textrm{FB}}$ or $\zeta^k_{\textrm{DR}}\frac{1+\rho L}{1+\rho m}$.

\subsection{Solving generalized equations}
As mentioned above, we are now interested in solving the prediction and correction steps, \textit{i.e.}, the generalized equations \eqref{eq:approx-gen-equation} and \eqref{eq:gen-equation}. That is, we need to minimize $h_k(\x) + g(x)$ for the prediction step, and $f(\x;t_{k+1}) + g(\x)$ for the correction step. For this, we can apply the operator splitting methods described in the previous part of this section.

However, the splitting methods are guaranteed to converge \textit{asymptotically} to the desired solution, but for practical reasons (finite length of the sampling period) we stop the algorithms after a fixed number of steps: $P$ for the prediction and $C$ for the correction phases. Therefore we obtain in both steps an approximation of the optimal solution.

Using the convergence results reported above, it is therefore possible to characterize the convergence rate of both splitting methods and quantify the reduction in the error after the prescribed number of steps.

We have now all the ingredients to fully describe the proposed prediction-correction algorithm, which is reported in Algorithm~\ref{alg:splitting-prediction-correction}. At every time $t_k$, we perform $P$ step of either the FBS or the DRS on the predicting generalized equation, 
\begin{equation*}
	\nabla h_k({\x}_{k+1|k}) + \partial g({\x}_{k+1|k}) \ni 0.
\end{equation*}
constructed for the prediction problem~\eqref{eq:approx-gen-equation} [cf. line~\ref{line1}]; this yields an approximate predictor $\tilde{\x}_{k+1|k}$. 

At time $t_{k+1}$, we observe the new cost function $f(\cdot;t_{k+1})$ [cf. line~\ref{line2}], and we perform $C$ steps of either the FBS or the DRS on the correcting generalized equation,
\begin{equation*}
	\nabla_{\x} f(\x;t_{k+1}) + \partial g(\x) \ni 0.
\end{equation*}
constructed for the correction problem~\eqref{eq:gen-equation} [cf. line~\ref{line3}]; this yields the approximate optimizer $\x_{k+1}$.  

\begin{algorithm}
\caption{Prediction-correction algorithm via splitting}
\label{alg:splitting-prediction-correction}
\begin{algorithmic}[1]
	\Require $\x_0$, step-size $\rho$, termination conditions $P$ and $C$, and specify a splitting method.
	\For{$k=0,1,\ldots$}
		\State // time $t_k$
		\State perform $P$ steps of the splitting with initial condition $\x_k$ by applying \eqref{eq:fbs} for FBS or \eqref{eq:drs} for DRS \label{line1}
		\State // time $t_{k+1}$
		\State observe the cost function $f(\cdot;t_{k+1})$\label{line2}
		\State perform $C$ steps of the splitting with initial condition the prediction $\tilde{\x}_{k+1|k}$ by applying \eqref{eq:fbs} for FBS or \eqref{eq:drs} for DRS\label{line3}
		\State set $\x_{k+1}$ equal to the last iterate of the splitting
	\EndFor
\end{algorithmic}
\end{algorithm}

\begin{remark}
We choose to use the same step-sizes for the prediction and correction steps, since both $f(\x;t_{k+1})$ and $h_k(\x)$ belong to $\mathcal{S}_{m,L}(\R^n)$ and thus the step-sizes can be chosen minimizing~\eqref{eq:zeta_fbs}~and~\eqref{eq:zeta_drs}.
\end{remark}

\section{Convergence Analysis}\label{sec:convergence}
Algorithm~\ref{alg:splitting-prediction-correction} entails the approximated solution of two optimization problems, therefore we do not expect the algorithm to converge exactly to the optimal trajectory. Instead, this section aims at proving the convergence of Algorithm~\ref{alg:splitting-prediction-correction} to a neighborhood of $\x^*(t_k) =: \x_k^*$ which depends on the sampling period $T_{\mathrm{s}}$.

We state the following results, which will be proved in the remainder of this Section and in the Appendix.

\begin{theorem}\label{th:linear-convergence}
Let Assumptions~\ref{as:first}-\ref{as:first-bis} hold, and choose the parameters of Algorithm~\ref{alg:splitting-prediction-correction} such that
\begin{equation}\label{eq:condition-th-1}
	\zeta(C) \left[ \zeta(P) + (\zeta(P) + 1) \frac{2L}{m} \right] < 1.
\end{equation}
Then the trajectory $\{\x_k\}_{k \in \mathbb{N}}$ generated by Algorithm~\ref{alg:splitting-prediction-correction} converges Q-linearly to a neighborhood of the optimal trajectory $\{\x_k^*\}_{k \in \mathbb{N}}$, and in particular
$$
	\limsup_{k \to \infty} \norm{\x_k - \x_k^*} = O(\zeta(C) T_{\mathrm{s}}),
$$
where function $\zeta(k)$ is either $\zeta^k_{\textrm{FB}}$ or $\zeta^k_{\textrm{DR}}\frac{1+\rho L}{1+\rho m}$, depending if one uses FBS or DRS in Algorithm~\ref{alg:splitting-prediction-correction}. 
\endstatement
\end{theorem}

\smallskip

\begin{theorem}\label{th:quadratic-convergence}
Let Assumptions~\ref{as:first},~\ref{as:first-bis},~and~\ref{as:second} hold, and select the prediction and correction horizons $P$ and $C$, and $\tau \in (0,1)$, such that $\zeta({P+C}) < \tau$.

\noindent Then there exist an upper bound for the sampling time $\bar{T}_{\mathrm{s}}$ and a convergence region $\bar{R}$ such that if $T_{\mathrm{s}} < \bar{T}_{\mathrm{s}}$ and $\norm{\x_0 - \x_0^*} < \bar{R}$, then
$$
	\limsup_{k \to \infty} \norm{\x_k - \x_k^*} = O(\zeta(C) T_{\mathrm{s}}^2) + O(\zeta({P+C}) T_{\mathrm{s}}).
$$
In particular, the bound for the sampling time and the convergence region are characterized by
$$
	\bar{T}_{\mathrm{s}} = \frac{\tau - \zeta(P+C)}{\zeta(C) (\zeta(P) + 1)} \left( \frac{C_0C_1}{m^2} + \frac{C_2}{m} \right)^{-1}
$$
and
$$
	\bar{R} = \frac{2m}{C_1} \left( \frac{C_0C_1}{m^2} + \frac{C_2}{m} \right) (\bar{T}_{\mathrm{s}} - T_{\mathrm{s}}).
$$
Function $\zeta(k)$ is either $\zeta^k_{\textrm{FB}}$ or $\zeta^k_{\textrm{DR}}\frac{1+\rho L}{1+\rho m}$, depending if one uses FBS or DRS in Algorithm~\ref{alg:splitting-prediction-correction}. 
\endstatement
\end{theorem}

\smallskip

\begin{remark}
Notice that in Theorem~\ref{th:linear-convergence} the asymptotical error bound depends linearly on the sampling time; this result agrees with the fact that the smaller $T_{\mathrm{s}}$ is, the better the algorithm can track the optimal trajectory.

Moreover, if bounds on the third order derivatives of $f$ are available, then we can characterize the convergence to a smaller neighborhood of the optimal trajectory. Indeed Theorem~\ref{th:quadratic-convergence} shows that the asymptotical error depends on the square of $T_{\mathrm{s}}$ and thus is tighter than the result of Theorem~\ref{th:linear-convergence}.
\end{remark}

\begin{remark}
Theorem~\ref{th:linear-convergence} holds globally, while Theorem~\ref{th:quadratic-convergence} holds only if the initial condition $\x_0$ is selected in a neighborhood $\bar{R}$ of the optimum. Moreover, notice that the larger we choose $T_\mathrm{s}$, the smaller $\bar{R}$ becomes, again in accordance with the fact that the closer we sample the problem the better we approximate the optimal trajectory.
\end{remark}

\subsection{Proof of Theorems~\ref{th:linear-convergence} and~\ref{th:quadratic-convergence}}

The full proof of Theorems~\ref{th:linear-convergence} and~\ref{th:quadratic-convergence} follows closely along the lines of \cite[Appendix~B]{simonetto2017prediction}, and it is reported in the Appendix for completeness. Here however, we give a sketch of the reasoning and a result on implicit solution mappings, which enables us to generalize~\cite{dontchev2014implicit} to our setting. 

Recall from Algorithm~\ref{alg:splitting-prediction-correction} that the following sources of error are present: \emph{(i)} the Taylor approximation used during the prediction step, and \emph{(ii-iii)} the early termination of the splitting method during both the prediction and correction steps. Thus, in order to evaluate the overall error committed by the proposed algorithm, i.e., $\norm{\x_k - \x_k^*}$, we need to derive a bound on the errors coming from \emph{(i)}, \emph{(ii)}, and \emph{(iii)}.

The proof proceeds as follows. We bound the Taylor approximation used during the prediction step by using a result in implicit solution mappings (see Theorem~\ref{th:sol-mapping}). We then bound the error coming from the early termination of the splitting method during prediction, by using the linear convergence rate of the splitting method (cf. Eq.s~\eqref{eq:convergence-splitting}~and~\eqref{eq:convergence-drs}). The next step is the derivation of a bound on the early termination of the splitting method during correction, which is analogous to the one during prediction. Finally, we put everything together and derive an error recursion, from which the claims.

Since the main novelty of the proof is the mentioned result in implicit solution mappings, we present this here. First of all, consider the parametrized generalized equation\footnote{In which the differentiation is taken w.r.t. to the $\x$.}
\begin{equation}\label{eq:param-gen-equation}
	\nabla \varphi(\x) + \partial \gamma(\x) \ni \p
\end{equation}
where $\p \in \R^d$ is a parameter, and define the \textit{solution mapping} of \eqref{eq:param-gen-equation} as
\begin{equation}\label{eq:solution-mapping}
	S(\p) = \{ \y\ |\ \nabla \varphi(\y) + \partial \gamma(\y) \ni \p \}.
\end{equation}
We can characterize the Lipschitz continuity of the solution mapping \eqref{eq:solution-mapping} using the following Theorem~\ref{th:sol-mapping}\footnote{This result extends Theorem 2F.6 in \cite{dontchev2014implicit}, which holds only if $\gamma$ is the indicator function of a convex set, and thus $\partial \gamma$ is a normal cone -- that is, it holds only for variational inequalities and not generalized equations. Theorem~\ref{th:sol-mapping} here resembles~\cite[Th.~1]{nesterov_smooth_2005}, although the latter is applied to a slightly different setting. }.

\begin{theorem}\label{th:sol-mapping}
Let $\varphi \in \mathcal{S}_{m,L}(\R^n)$ and $\gamma \in \Gamma_0(\R^n)$, then the solution mapping \eqref{eq:solution-mapping} of the parameterized equation \eqref{eq:param-gen-equation} is single-valued and $m^{-1}$-Lipschitz continuous. \endstatement
\end{theorem}

\begin{proof}
The proof is divided in two parts.

\paragraph{Single-valuedness}
The solutions of \eqref{eq:param-gen-equation} are the solutions of
$$
	\y^* = \argmin_\y \{ \varphi(\y) + \gamma(\y) - \langle \p,\y \rangle \}.
$$
But since $\varphi$ is strongly convex, so is the whole objective function once the parameter $\p$ is fixed, which implies that the problem has a unique solution for each $\p$.

\paragraph{Lipschitz continuity}
Let $\y = S(\p)$ and $\w = S(\mathbf{q})$ for $\p,\mathbf{q} \in \R^n$. In order to prove Lipschitz continuity we need to show that $\norm{\y - \w} \leq m^{-1} \norm{\p - \mathbf{q}}$. By the definition of solution mapping it holds $\p - \nabla \varphi(\y) \in \partial \gamma(\y)$ and $\mathbf{q} - \nabla \varphi(\w) \in \partial \gamma(\w)$; therefore by the definition of the subdifferential it must hold
\begin{align}
	\langle \z - \y, \p - \nabla \varphi(\y) \rangle & \leq \gamma(\z) - \gamma(\y) \label{eq:y-subdiff} \\
	\langle \z - \w, \mathbf{q} - \nabla \varphi(\w) \rangle & \leq \gamma(\z) - \gamma(\w) \label{eq:w-subdiff}
\end{align}
for any $\z \in \R^n$. In particular, inequality \eqref{eq:y-subdiff} is satisfied for $\z = \w$, and \eqref{eq:w-subdiff} for $\z = \y$; changing sign in the latter, the following inequalities hold
$$
	\langle \w - \y, \p - \nabla \varphi(\y) \rangle \leq \gamma(\w) - \gamma(\y) \leq \langle \w - \y, \mathbf{q} -\nabla\varphi(\w) \rangle
$$
or, equivalently,
$$
	\langle \w - \y, \nabla \varphi(\w) - \nabla \varphi(\y) \rangle \leq \langle \w - \y, \mathbf{q} - \p \rangle.
$$
Since $\nabla \varphi$ is the derivative of an $m$-strongly convex function, then it is strongly monotone \cite{bauschke2017convex}, in the sense that for any $\x,\y \in \R^n$ it holds
$$
	\langle \x - \y, \nabla \varphi(\x) - \nabla \varphi(\y) \rangle \geq m \norm{\x - \y}^2.
$$
With this result and the fact that $\langle \w - \y, \mathbf{q} - \p \rangle \leq \norm{\w - \y} \norm{\mathbf{q} - \p}$, it follows then
\begin{align*}
	m \norm{\w - \y}^2 \leq \langle \w - \y, \nabla \varphi(\w) - & \nabla \varphi(\y) \rangle \\ & \leq \norm{\w - \y} \norm{\mathbf{q} - \p}
\end{align*}
which proves the desired result.
\end{proof}

\section{Simulations}\label{sec:simulation}

To assess the practical performance of the proposed methods, we present some simple yet realistic numerical simulations for a \emph{leader following} case problem.

\subsection{Problem formulation}
We are interested in solving a leader following problem for a group of $N$ robots in a \emph{rigid formation}. In particular, the leader (or $0$-th robot) follows freely a trajectory $\{ \x_k^0 \}_{k \in \mathbb{N}}$ on the plane, and the follower robots (labeled $1$ through $N$) must: \textit{(i)} follow the leader, while \textit{(ii)} keeping a fixed formation. Hereafter, we denote with $\x^i \in \R^2$ the position of the $i$-th robot, and define $\x = [ (\x^0)^\top, (\x^1)^\top, \ldots, (\x^N)^\top ]^\top \in \R^{2(N+1)}$.

The followers, however, have access only to a noisy and partial measurement of the leader's position, which is given by $z_{k+1}^i = \mathbf{v}_i^\top \x_{k+1}^0 + n^i_{k+1}$ where $\mathbf{v}_i \in \{ [1, 0]^\top, [0, 1]^\top \}$ and the noise is Gaussian $n^i \sim \mathcal{N}(0, \sigma_i)$.

Inspired by the formulation presented in~\cite{dixit2018online}, the idea is then to estimate $\x^0$ by solving a least squares problem as
$$
	f(\x;t_{k+1}) = \sum_{i=1}^N \frac{1}{2} \left( z_{k+1}^i - \mathbf{v}_i^\top \x^0 \right)^2 + \frac{\lambda}{2} \norm{\x - \x_k}_2^2.
$$

The formation is to be kept constant, with the leader robot in the center and the followers at a distance $d$ from it, with equal angles between them. By defining a reference frame centered in the leader, with $x$-axis pointing in the direction of motion, the position of the $i$-th robot must satisfy
\begin{align*}
	\x^i - \x^0 = d \begin{bmatrix}
		\cos(2\pi(i-1)/N) \\
		\sin(2\pi(i-1)/N)
	\end{bmatrix}.
\end{align*}
These constraints are collected in the linear system $\A \x = \bv$:
$$
\footnotesize
	\A = \begin{bmatrix}
		-I & I & 0 & \dots & 0 \\
		-I & 0 & I & \dots & 0 \\
		\vdots & \vdots &  & \ddots \\
		-I & 0 & \dots & 0 & I
	\end{bmatrix}, \quad
	\bv = d \begin{bmatrix}
		\vdots \\
			\cos\left(\frac{2\pi(i-1)}{N}\right) \\
			\sin\left(\frac{2\pi(i-1)}{N}\right)
		 \\
		\vdots
	\end{bmatrix}
$$
and imposed by choosing $g(\x) = \iota_{\A\x = \bv}(\x)$.

Notice that the problem is solved by a fusion center that collects the measurements and then dictates the trajectory that the followers need to perform.


\subsection{Numerical results}
The simulations where carried out for a formation of $N=10$ followers, placed on the unitary circle around the leader ($d=1$). The noise variance was equal to $\sigma_i = 0.1$ for all robots, and six followers had $v_i = [1, 0]^\top$ while the other four had $v_i = [0, 1]^\top$. The regularization constant was $\lambda = 10$, for the FBS we set $\rho = 1/L$, while for DRS we set $\rho = 0.08$ (both to maximize the respective contraction ratio). The sampling time was set to $T_\mathrm{s} = 0.1 \si{s}$. The trajectory of the leader is a Lissajous curve with ratio 1:3 and spans an area of $[-3,3]^2$ with a period of about $40\si{s}$. Our performance metric is the tracking error, defined as $
	E_k = \norm{\x_k - \x_k^*}$. 

%
%
%

First of all, we evaluate the performance of FBS vs. DRS with $P=0, C=5$ (Figure~\ref{fig:tracking}-top); as we see, FBS outperforms DRS for this problem setting (but this is far from a general statement).

Second, we look at FBS for different values of the prediction horizon $P$, and with fixed correction horizon $C=5$ (Figure~\ref{fig:tracking}-middle). The presence of at least one prediction step improves the accuracy of the algorithm, thus justifying the choice of a prediction-correction scheme. We see also that, since we compute the prediction by using a backward derivative approximation of the time derivative of the gradient, and noise is added, the case $P=5$ hits the noise ``floor''. 

   \begin{figure}
      \centering
      \includegraphics[width=\columnwidth, trim=0cm 1cm 1cm 1cm, clip=on]{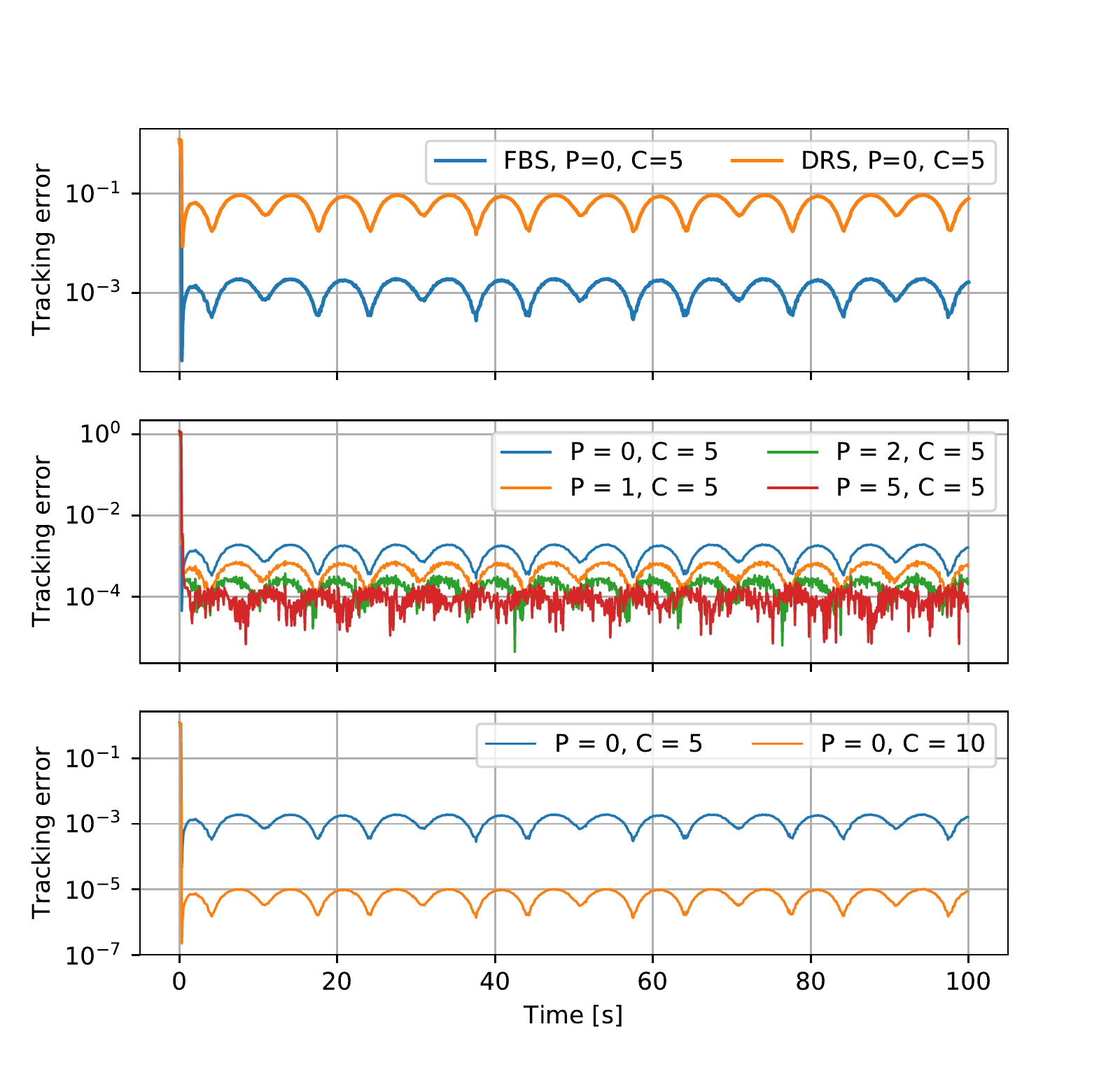}
      \caption{Tracking error in different settings. }
      \label{fig:tracking}
   \end{figure}

Figure~\ref{fig:tracking}-bottom depicts instead the error for different values of the correction steps for FBS, with $P=0$.

The smaller the number of correction steps is, the worse the asymptotical error. This result shows that augmenting $C$ seems to be more effective in reducing the tracking error than augmenting $P$, in this particular case. 
Note that however, correction steps are more computational expensive than prediction steps (since the latter are computed on a quadratic cost function). 

We then evaluate the effect of the sampling time on the performance of the proposed algorithm by computing the asymptotical error for different values of $T_\mathrm{s}$, $P$ and $C$, which is computed as $\max_{k > 2D/3} E_k$ with $D=100\si{s}$ the duration of the simulation. In this case (and for verification purposes only) we let the noise be depended on the sampling period as $\sigma_i = .01 T_\mathrm{s}$, so to remove its effect from the numerical results. (Naturally this is not the case in practice and $T_\mathrm{s}$ needs to be chosen to limit the effect of the noise on the prediction, as well as maintain an acceptable update frequency).  

   \begin{figure}
      \centering
      \includegraphics[width=\columnwidth,trim=0cm 0cm 1cm 1cm, clip=on]{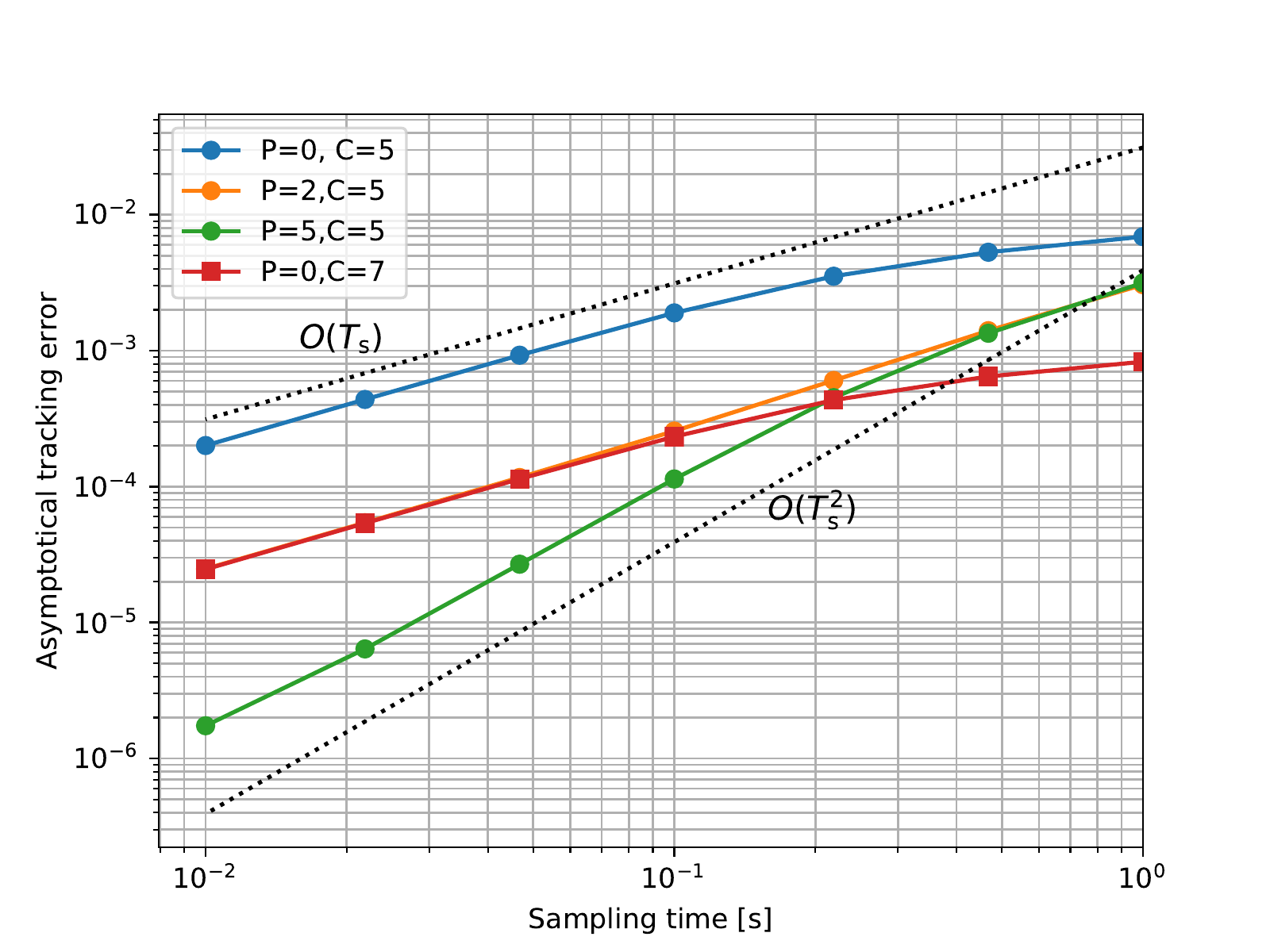}
      \caption{Asymptotical error of the FBS as a function of $T_\mathrm{s}$.}
      \label{fig:asymptotical}
   \end{figure}

In Figure~\ref{fig:asymptotical} shows that the larger the sampling time is, the worse the performance of the proposed algorithm is, since the sequence of problems \eqref{eq:ti-problem} ``tracks'' the original problem worse. Notice that, as expected from the theorems, we obtain the $O(T_\mathrm{s})$ and $O(T_\mathrm{s}^2)$ dependencies for the asymptotical error, further justifying the use of prediction steps. 

\section{Conclusions and Future Directions}\label{sec:conclusions}
In this work, we presented an algorithm to solve time-varying optimization problems that applies splitting operators in conjunction with a prediction-correction scheme. We described two theoretical results that guarantee the convergence of the sequence computed by the algorithm to a neighborhood of the optimal solutions that depends on the sampling time, or its square. Finally we described some numerical results obtained by applying the algorithm to a leader following problem in robotics. 

Future work will focus on extending the algorithm to solve problems characterized by weaker conditions, such as removing strong convexity. Moreover, the application to distributed optimization problems will be explored, as well as further acceleration via e.g., \cite{Fazlyab2017}.

\newpage


\appendix[Proof of Lemma~\ref{lem:convergence-drs}]

For the DRS it is possible to prove \cite{taylor2017convex} that the sequence $\{\z(k)\}_{k \in \mathbb{N}}$ converges Q-linearly with rate ${\zeta}_\textrm{DR}$ to $\z^*$, that is
\begin{equation}\label{eq:lin-convergence}
	\norm{\z(k+1) - \z^*} \leq {\zeta}_\textrm{DR} \norm{\z(k) - \z^*}.
\end{equation}
By the assumptions on $\varphi$ it follows \cite{giselsson2014diagonal} that $\prox_{\rho \varphi}(\x)$ is $1/(1+\rho m)$-contractive and $1/(1+\rho L)$-strongly monotone; therefore by using the fact that $\x(k) = \prox_{\rho \varphi}(\z(k))$, we obtain,
\begin{multline}
	\norm{\x(k)-\x^*} \leq \frac{1}{1+\rho m}\,\norm{\z(k) - \z^*} \leq  \\ \frac{1}{1+\rho m}\,{\zeta}_\textrm{DR}^k \norm{\z(0) - \z^*} \leq  \frac{1+\rho L}{1+\rho m}\,{\zeta}_\textrm{DR}^k \norm{\x(0) - \x^*},
\end{multline} 
from which the claim. \hfill$\blacksquare$
%
%
%

\appendix[Proof of Theorems \ref{th:linear-convergence} and \ref{th:quadratic-convergence}]

\subsection{Taylor approximation error}
By making use of Theorem~\ref{th:sol-mapping}, we derive an upper bound to the error introduced by the Taylor approximation. Consider the generalized equation \eqref{eq:approx-gen-equation} that characterizes the prediction step, and define the following functions:
\begin{align*}
	\Psi(\y) =& \nabla_{\x} f(\x_{k+1}^*;t_{k+1}) + \\ &+ \nabla_{\x\x} f(\x_{k+1}^*;t_{k+1}) (\y - \x_{k+1}^*) + \partial g(\y)
\end{align*}
and
\begin{align*}
	\psi(\y) =& \nabla_{\x} f(\x_k;t_k) + \nabla_{\x\x} f(\x_k;t_k) (\y - \x_k) + \\ &+ T_{\mathrm{s}} \nabla_{t\x} f(\x_k;t_k) - \Big[ \nabla_{\x} f(\x_{k+1}^*;t_{k+1}) + \\ & +\nabla_{\x\x} f(\x_{k+1}^*;t_{k+1}) (\y - \x_{k+1}^*) \Big].
\end{align*}
With these definitions the prediction problem \eqref{eq:approx-gen-equation} is equivalent to solving $(\psi + \Psi)(\x_{k+1|k}) \ni 0$, while the problem \eqref{eq:gen-equation} is equivalent to $\Psi(\x_{k+1}^*) \ni 0$, and therefore it holds $(\psi + \Psi)(\x_{k+1}^*) \ni \psi(\x_{k+1}^*)$.

Consider now the parametrized generalized equation
$$
	(\psi + \Psi)(\x) \ni \p
$$
with solution mapping $S(\p)$. By the results above we have $S(\p) = \x_{k+1|k}$ when $\p = 0$, and $S(\mathbf{q}) = \x_{k+1}^*$ when $\mathbf{q} = \psi(\x_{k+1}^*)$. Employing now Theorem~\ref{th:sol-mapping} it follows
$$
	\norm{\x_{k+1|k} - \x_{k+1}^*} = \norm{S(\p) - S(\mathbf{q})} \leq m^{-1} \norm{\psi(\x_{k+1}^*)}.
$$

If Assumptions~\ref{as:first}-\ref{as:first-bis} hold, then following the steps detailed in \cite[Appendix~B]{simonetto2017prediction} it is possible to derive the bound
\begin{equation}
	\norm{\x_{k+1|k} - \x_{k+1}^*} \leq \frac{2L}{m} \norm{\x_k - \x_k^*} + 2T_{\mathrm{s}} \frac{C_0}{m} (1 + L/m)
\end{equation}
for the error due to the Taylor approximation at the prediction step.

If moreover Assumption~\ref{as:second} holds then the following quadratic bound can be derived instead
\begin{align}
\begin{split}
	\norm{\x_{k+1|k} - \x_{k+1}^*} \leq& \frac{C_1}{2m} \norm{\x_k - \x_k^*}^2 + \\
	&+ T_{\mathrm{s}} \left( \frac{C_0C_1}{m^2} + \frac{C_2}{m} \right) \norm{\x_k - \x_k^*} + \\
	& + \frac{T_{\mathrm{s}}^2}{2} \left[ \frac{C_0^2C_1}{m^3} + \frac{2C_0C_2}{m^2} + \frac{C_3}{m} \right].
\end{split}
\end{align}

\subsection{Early termination error}
In Algorithm~\ref{alg:splitting-prediction-correction} both during the prediction and the correction steps, the selected splitting method is halted after $P$ and $C$ steps, respectively.

\paragraph{FBS} Making use of the convergence results reported in \eqref{eq:convergence-splitting} for the FBS, the prediction error after $P$ steps will be bounded as
\begin{equation}\label{eq:prediction-error}
	\norm{\tilde{\x}_{k+1|k} - \x_{k+1|k}} \leq \zeta_{\textrm{FB}}^P \norm{\x_k - \x_{k+1|k}}
\end{equation}
while the correction error after $C$ steps will be bounded as
\begin{equation}\label{eq:correction-error}
	\norm{\x_{k+1} - \x_{k+1}^*} \leq \zeta_{\textrm{FB}}^C \norm{\tilde{\x}_{k+1|k} - \x_{k+1}^*}.
\end{equation}

\paragraph{DRS} Applying instead Lemma~\ref{lem:convergence-drs} for the DRS it follows that the prediction error is bounded as
\begin{equation*}
	\norm{\tilde{\x}_{k+1|k} - \x_{k+1|k}} \leq \zeta(P) \norm{\x_k - \x_{k+1|k}}
\end{equation*}
where $\zeta(P) = \zeta_{\textrm{DR}}^P (1+\rho L)/(1+\rho m)$, while the correction error as
\begin{equation*}
	\norm{\x_{k+1} - \x_{k+1}^*} \leq \zeta(C) \norm{\tilde{\x}_{k+1|k} - \x_{k+1}^*}.
\end{equation*}

In the following, we will indicate with function $\zeta(k)$, either $\zeta_{\textrm{FB}}^k$ for FBS, or $\zeta_{\textrm{DR}}^P (1+\rho L)/(1+\rho m)$ for DRS. 

\subsection{Overall error bound}
Recall that our aim is to derive a bound for the error $\norm{\x_{k+1} - \x_{k+1}^*}$ that depends on the error at the previous time instant $\norm{\x_k - \x_k^*}$.

In order to do so we need one last result. Under Assumptions~\ref{as:first}-\ref{as:first-bis}, it is possible to see that the solution mapping $t_k \mapsto \x^*(t_k)$ satisfies Theorem~\ref{th:sol-mapping}, and thus that
\begin{align}\label{eq:optimal-trajectory-ineq}
\begin{split}
	\norm{\x^*(t_{k+1}) - \x^*(t_k)} &\leq \frac{1}{m} \norm{\nabla_{t\x} f(\x;t)} (t_{k+1} - t_k) \\ & \leq \frac{C_0T_{\mathrm{s}}}{m}.
\end{split}
\end{align}
We are now ready to find the desired error bound.

Using the results above, we can now derive the following chain of inequalities
\begin{align*}
	& \norm{\tilde{\x}_{k+1|k} - \x_{k+1}^*} \\
	& \leq \norm{\tilde{\x}_{k+1|k} - \x_{k+1|k}} + \norm{\x_{k+1|k} - \x_{k+1}^*} \\
	& \leq \zeta(P) \norm{\x_k - \x_{k+1|k}} + \norm{\x_{k+1|k} - \x_{k+1}^*} \\
	& \leq \zeta(P) (\norm{\x_k -\x_k^*} + \norm{\x_k^* - \x_{k+1}^*} + \\
	& \qquad\qquad + \norm{\x_{k+1}^* - \x_{k+1|k}}) + \norm{\x_{k+1|k} - \x_{k+1}^*} \\
	& \leq \zeta(P) \norm{\x_k - \x_k^*} + (\zeta(P) + 1) \norm{\x_{k+1|k} - \x_{k+1}^*} + \\
	& \qquad\qquad + \zeta(P) \norm{\x_k^* - \x_{k+1}^*}
\end{align*}
which gives a bound for the prediction error. Moreover, employing~\eqref{eq:correction-error} it holds
\begin{multline*}
	\norm{\x_{k+1} - \x_{k+1}^*} \leq \zeta(C) \Big[ \zeta(P) \norm{\x_k - \x_k^*} + \\
	 + (\zeta(P) + 1) \norm{\x_{k+1|k} - \x_{k+1}^*} + \zeta(P) \norm{\x_k^* - \x_{k+1}^*} \Big]
\end{multline*}
and, under Assumptions~\ref{as:first}-\ref{as:first-bis}, we can write
$$
	\norm{\x_{k+1} - \x_{k+1}^*} \leq \eta_2 \norm{\x_k - \x_k^*}^2 + \eta_1 \norm{\x_k - \x_k^*} + \eta_0 
$$
where
\begin{align*}
	\eta_2 &= 0 \\
	\eta_1 &= \zeta(C) \left[ \zeta(P) + (\zeta(P) + 1) \frac{2L}{m} \right] \\
	\eta_0 &= \zeta(C) \left[ 2(\zeta(P)+1) \left( 1+\frac{L}{m} \right) + \zeta(P) \right] \frac{C_0 T_{\mathrm{s}}}{m}.
\end{align*}
If moreover Assumption~\ref{as:second} holds as well, then we have
\begin{align*}
	\eta_2 &= \zeta(C) (\zeta(P) + 1) \frac{C_1}{2m} \\
	\eta_1 &= \zeta(C) \left[ \zeta(P) + T_{\mathrm{s}} (\zeta(P) + 1) \left( \frac{C_0C_1}{m^2} + \frac{C_2}{m} \right) \right] \\
	\eta_0 &= \zeta(C) \Bigg[ \zeta(P) \frac{T_{\mathrm{s}}C_0}{m} + \\
	& \qquad + (\zeta(P) + 1) \frac{T_{\mathrm{s}}^2}{2} \Bigg( \frac{C_0^2C_1}{m^3} + \frac{2C_0C_2}{m^2} + \frac{C_3}{m} \Bigg) \Bigg].
\end{align*}

\subsection{Convergence proof}
As the previous section showed, under Assumptions~\ref{as:first}-\ref{as:first-bis} we can characterize the dynamics of the error with the inequality
$$
	\norm{\x_{k+1} - \x_{k+1}^*} \leq \eta_1 \norm{\x_k - \x_k^*} + \eta_0.
$$
Therefore, for the error to converge, it is necessary to choose the algorithm parameters in such a way that $\eta_1 < 1$. In particular, given the definition of $\eta_1$ and the rate $\zeta$, we need to make sure that the prediction and correction horizons, $P$ and $C$, and the step-size $\rho$, satisfy the condition $\eta_1 < 1$.

If this is the case, than it is clear that by the definition of $\eta_0$, Theorem~\ref{th:linear-convergence} holds.

Suppose now that Assumption~\ref{as:second} holds as well. We have convergence of the algorithm if the error performed at each iteration does not exceed the error of the previous iteration, that is if there exists $\tau < 1$ such that
$$
	\eta_2 \norm{\x_k - \x_k^*}^2 + \eta_1 \norm{\x_k - \x_k^*} + \eta_0 \leq \tau \norm{\x_k - \x_k^*} + \eta_0.
$$
Therefore the error converges if $\eta_1 < \tau < 1$ which implies
\begin{multline*}
	\zeta({P+C}) < \tau, \\ T_{\mathrm{s}} < \frac{\tau - \zeta({P+C})}{\zeta(C)(\zeta(P) + 1)} \left( \frac{C_0C_1}{m^2} + \frac{C_2}{m} \right)^{-1} =: \bar{T}_{\mathrm{s}},
\end{multline*}
and if the initial condition $\x_0$ is chosen such that
$$
	\norm{\x_0 - \x_0^*} \leq \frac{\tau - \eta_1}{\eta_2} =: \bar{R}.
$$
Using now the fact that
\begin{align*}
	\norm{\x_{k+1} - \x_{k+1}^*} &\leq \eta_2 \norm{\x_k - \x_k^*}^2 + \eta_1 \norm{\x_k - \x_k^*} + \eta_0 \\
	& \leq \tau \norm{\x_k - \x_k^*} + \eta_0
\end{align*}
we can compute
$$
	\norm{\x_k - \x_k^*} \leq \tau^k \norm{\x_0 - \x_0^*} + \eta_0 \frac{1 - \tau^k}{1 - \tau}
$$
which proves Theorem~\ref{th:quadratic-convergence} when $k \to \infty$. \hfill$\blacksquare$



\end{document}